\documentclass[a4paper, twoside, 12pt]{amsart}

\usepackage[T1]{fontenc}              
\usepackage[latin1]{inputenc}         
\usepackage{geometry}                 

\usepackage{amsmath,amsfonts,amssymb} 
\usepackage{bbm}
\usepackage[capitalize]{cleveref}
\usepackage{amsthm}                   
\usepackage{graphicx}                 
\usepackage{subfig}                   
\usepackage{listings}                 
\usepackage{enumerate}
\usepackage{tikz-cd}
\usepackage{rotating}
\usepackage{soul} 
\usepackage{todonotes}
\usepackage{mathtools}

\newtheorem{thmu}{Theorem}

\newtheorem{lem}{Lemma}[section]
\newtheorem{prop}[lem]{Proposition}
\newtheorem{cor}[lem]{Corollary}
\newtheorem{thm}[lem]{Theorem}
\theoremstyle{definition}
\newtheorem{defin}[lem]{Definition}
\newtheorem*{acknowledgements}{Acknowledgements}
\newtheorem{remark}[lem]{Remark}
\newtheorem{example}[lem]{Example}
\newtheorem*{motivating question}{Motivating question}
\newtheorem*{setup}{Setup}

\DeclareMathOperator{\rad}{rad}

\DeclareMathOperator{\Tr}{Tr}

\DeclareMathOperator{\Thick}{Thick}

\DeclareMathOperator{\bbL}{\mathbb{L}}

\DeclareMathOperator{\proj}{proj}

\DeclareMathOperator{\gldim}{gldim}
\DeclareMathOperator{\Ext}{Ext}

\DeclareMathOperator{\gr}{gr}
\DeclareMathOperator{\stgrmodu}{\setul{0.6ex}{0.1ex}
\textup{\ul{gr}}}

\DeclareMathOperator{\Hom}{Hom}
\DeclareMathOperator{\RHom}{\mathbb{R}Hom}
\DeclareMathOperator{\REnd}{\mathbb{R}End}
 
\DeclareMathOperator{\End}{End}

\DeclareMathOperator{\modu}{mod}

\DeclareMathOperator{\Hm}{H}
\DeclareMathOperator{\HH}{HH}

\DeclareMathOperator{\op}{op}

\DeclareMathOperator{\perf}{perf}

\DeclareMathOperator{\De}{\mathcal{D}}

\title{Higher Koszul algebras and the \textbf{(Fg)}-condition}
\author{Johanne Haugland and Mads Hustad Sand\o y}

\begin{document}

\keywords{Higher Koszul algebra, the \textbf{(Fg)}-condition, support varieties, higher homological algebra, trivial extension, preprojective algebra, $n$-representation infinite algebra, $n$-representation tame algebra, dimer algebra, dimer model, $A_\infty$-algebra}
\subjclass[2020]{16E40, 16S37, 16W50, 16G20, 16G60, 18G80}

\address{Department of mathematical sciences, NTNU, NO-7491 Trondheim, Norway}
\email{johanne.haugland@ntnu.no}
\email{madshs@gmail.com}

\begin{abstract}
Determining when a finite dimensional algebra satisfies the finiteness property known as the (\textbf{Fg})-condition is of fundamental importance in the celebrated and influential theory of support varieties. We give an answer to this question for higher Koszul algebras, generalizing a result by Erdmann and Solberg. This allows us to establish a strong connection between the (\textbf{Fg})-condition and higher homological algebra, which significantly extends the classes of algebras for which it is known whether the (\textbf{Fg})-condition is satisfied. In particular, we show that the condition holds for an important class of algebras arising from consistent dimer models.
\end{abstract}

\maketitle

\tableofcontents

\section{Introduction}

The influential theory of support varieties for modules over group algebras of finite groups was introduced in \cite{Carlson1,Carlson2}, using the maximal ideal spectrum of the group cohomology ring. Analogous fruitful theories have later been established in different areas, e.g.\ for restricted Lie algebras \cite{Friedlander-Parshall}, cocommutative Hopf algebras \cite{Friedlander-Suslin'97} and complete intersections \cite{Avramov-Buchweitz}.

The general investigation of support varieties for arbitrary finite dimensional algebras was initiated in \cite{SS04}. These varieties are defined in terms of the action of the Hochschild cohomology ring on the $\Ext$-algebra of modules. Given that the crucial finiteness property known as the (\textbf{Fg})\emph{-condition} (see \cref{subsec: fg-condition}) is satisfied, these support varieties have been shown to encode important homological behaviour, similarly as in the classical setting of modular representations of finite groups. In particular, this includes being able to show that an algebra is wild if the complexity of the projective resolution of its simple modules is greater than two \cite{BS10b}. Moreover, if the algebra is also assumed to be self-injective, one has that modules in the same component of the Auslander--Reiten quiver have the same variety \cite{SS04}, and one obtains a generalization of Webb's theorem in the form of \cite[\mbox{Theorem 5.6}]{EHTSS}, meaning essentially that one can determine a nice list of possible tree classes of the components of the stable Auslander--Reiten quiver of the algebra.

Determining whether the (\textbf{Fg})-condition holds for a given class of finite dimensional algebras is hence fundamentally important for the study of support varieties. This leads to the following motivating question.

\begin{motivating question}
When does a finite dimensional algebra satisfy the (\textbf{Fg})-condition?
\end{motivating question}

The question above has attracted significant attention. In particular, the (\textbf{Fg})-condition has been shown to be invariant under several forms of equivalences including derived equivalence \cite{Kulshammer-Psaroudakis-Skartsaeterhagen}, separable equivalence \cite{B23}, and stable equivalence of Morita type with levels \cite{Skartsaeterhagen'16}. In addition, it is known that various ways of constructing new algebras from old ones preserve the (\textbf{Fg})-condition. Namely, forming skew group algebras and coverings preserve the condition up to some assumptions on the characteristic \cite{B23,Sandoey24}, and tensor products of algebras that satisfy the condition must themselves satisfy the condition \cite{Bergh-Oppermann'08}. 

Although it is known that not every finite dimensional algebra satisfies the (\textbf{Fg})-condition, the property has been shown to hold for several classes of algebras. On the one hand, it is known to hold for group algebras \cite{Evens'61, Venkov'59}, universal enveloping algebras of restricted Lie algebras, and more generally for finite dimensional cocommutative (graded) Hopf algebras \cite{Drupieski'16,Friedlander-Suslin'97}. Note that the proofs in all these cases leverage the powerful assumption of working with a cocommutative (graded) Hopf algebra. 

On the other hand, when one is not necessarily dealing with a cocommutative (graded) Hopf algebra, much less is known. Nevertheless, the (\textbf{Fg})-condition has been investigated in specific cases like for self-injective algebras of finite representation type \cite{Green-et-al-2}, for monomial algebras \cite{Dotsenko-Gelinas-Tamaroff'23,Jawad-Snashall-Taillefer'24}, for quantum complete intersection algebras \cite{Bergh-Oppermann'08}, for Koszul duals of certain classes of Artin--Schelter regular \mbox{algebras \cite{Itaba'19}}, and for self-injective radical cube zero algebras \cite{Erdmann-Solberg-2,Erdmann-Solberg-1,Said'15, Sandoey24}. In each of these cases, subclasses for which the (\textbf{Fg})-condition holds have been specifically identified. 

An important result by Erdmann and Solberg gives a characterization of when a finite dimensional Koszul algebra satisfies the (\textbf{Fg})-condition in terms of a criterion on the associated Koszul dual. More precisely, the (\textbf{Fg})-condition holds for such an algebra if and only if the Koszul dual is finitely generated over its graded center which is also noetherian \cite[Theorem 1.3]{Erdmann-Solberg-1}. It should be noted that this result is what allows for the classification of weakly symmetric algebras with radical cube zero satisfying the (\textbf{Fg})-condition obtained in \cite{Erdmann-Solberg-1}. Moreover, the result has later been applied to extend the classification to all self-injective algebras with radical cube zero \cite{Said'15,Sandoey24}. 

In this paper we investigate the motivating question from the viewpoint of higher homological algebra. The foundation for this approach is provided in \cite{HS}, where the authors introduce \emph{$n$-$T$-Koszul algebras} (see \cref{subsec: background n-T}) as a higher-dimensional analogue of classical Koszul algebras. This generalizes the notion of $T$-Koszul algebras from \cite{Green-Reiten-Solberg,Madsen}, where Koszulity is formulated with respect to a tilting module $T$, but the rigidity condition now additionally depends on a positive \mbox{integer $n$}.

In the first main result of this paper, given as \cref{thm: intro1} below, we prove that the characterization of classical finite dimensional Koszul algebras satisfying the (\textbf{Fg})-condition from \cite[Theorem 1.3]{Erdmann-Solberg-1} extends to the significantly bigger class of \mbox{$n$-$T$-Koszul} algebras. This provides a full answer to the motivating question for the class of higher Koszul algebras. For the definition of the $n$-$T$-Koszul dual $\Lambda^!$ of an $n$-$T$-Koszul algebra $\Lambda$, see \cref{def: n-T-Koszul dual}.

\begin{thmu}[see \cref{thm: characterisation for n-T}]\label{thm: intro1}
Let $\Lambda$ be a finite dimensional $n$-$T$-Koszul algebra. Then $\Lambda$ satisfies the \emph{(\textbf{Fg})}-condition if and only if the graded center $Z_{\gr}(\Lambda^!)$ is noetherian and $\Lambda^!$ is module finite over $Z_{\gr}(\Lambda^!)$.
\end{thmu}

The key idea in the proof of \cref{thm: intro1} is to employ work by Briggs and G{\'e}linas in the setup of $A_\infty$-algebras \cite{Briggs-Gelinas}. In order to get access to this theory, we demonstrate in \cref{thm: A_infty-structure on Koszul dual} that the dual of an $n$-$T$-Koszul algebra is indeed the cohomology of a formal $A_\infty$-algebra.

One important consequence of \cref{thm: intro1} is that it enables us to take advantage of the connections between higher Koszul algebras and higher homological algebra that are established in \cite{HS}. The class of $n$-hereditary algebras is introduced in \cite{HIO14,Iya11, Iyama-Oppermann} as a higher analogue of classical hereditary algebras from the viewpoint of higher Auslander--Reiten theory. These algebras have received significant attention \cite{DI20,Dramburg-Gasanova24b,HJS,HI11b, HI11,HKV25,ST,Vas19,Vas23} and have been shown to relate to many different areas of mathematics \cite{Amiot-Iyama-Reiten,DJL21,DJW19,EP12,HIMO14,HM24,IW14,JKM24,Nakajima,OT12}. The class of $n$-hereditary algebras splits up into \emph{$n$-representation finite} and \emph{$n$-representation infinite} algebras, coinciding with the classical notions of representation finite and representation infinite hereditary algebras in the case $n=1$. We note that while $n$-representation infinite algebras play the most important role in this paper, there are also connections between the (\textbf{Fg})-condition and the theory of $n$-representation finite algebras as outlined in \cref{rem: n-RF case}.

Our second main result, given as \cref{thm: intro2} below, highlights the significance of $n$-representation infinite algebras in the theory of support varieties. Classically, one can determine whether a hereditary algebra is tame by checking if its preprojective algebra is a noetherian algebra over its center; see \cref{prop: tame iff 1-tame}. \cref{thm: intro2} is obtained by combining a higher version of this with \cref{thm: intro1} and a characterization result for graded symmetric higher Koszul algebras from \cite{HS}. 

\begin{thmu}[see \cref{cor: fg}]\label{thm: intro2}
Let $\Lambda$ be a graded symmetric finite dimensional algebra of highest degree $1$ with $\Lambda_0$ an $n$-representation infinite algebra. Then $\Lambda$ satisfies the \emph{(\textbf{Fg})}-condition if and only if $\Lambda_0$ is $n$-representation tame.
\end{thmu}

As applications, we establish that the (\textbf{Fg})-condition holds for large classes of algebras for which it was not previously known; see \cref{subsec: applications and examples}. This includes trivial extensions of $2$-representation infinite algebras obtained from \mbox{\emph{dimer models}} on the torus. Dimer models and their associated dimer algebras are central notions in mathematics and physics that first arose in the field of statistical mechanics and which have later been intensively studied in relation to string theory \mbox{\cite{FHVW,AK05,Kennaway}}. In mathematics, this is of particular importance in algebraic geometry, as Jacobian algebras obtained from dimer models provide examples of so-called non-commutative crepant resolutions; see \cite{vandenBergh'23}. 

By combining \cref{thm: intro2} with work of Nakajima \cite{Nakajima}, we obtain the result below. 

\begin{thmu}[see \cref{thm: Fg and consistent dimer algebras}]\label{thm: intro3}
Let $\Gamma$ be a dimer algebra associated to a consistent dimer model, and assume that the dimer model has a perfect matching inducing a grading such that $A \coloneq \Gamma_0$ is finite dimensional. Then the trivial extension $\Delta A$ of $A$ satisfies the \emph{(\textbf{Fg})}-condition. 
\end{thmu}

The paper is structured as follows. In \cref{sec: preliminaries} we give an overview of some definitions and results that are needed in the rest of the paper. This includes an introduction to the (\textbf{Fg})-condition as well as necessary background concerning \mbox{$A_\infty$-algebras}. \cref{sec: general results} presents some general results providing key steps towards the proof of \cref{thm: intro1}. In \cref{sec: higher Koszul} we investigate when higher Koszul algebras satisfy the (\textbf{Fg})-condition and prove \cref{thm: intro1}. Building on this, we establish the strong connection between the (\textbf{Fg})-condition and higher Auslander--Reiten theory given in \cref{thm: intro2}. In \cref{subsec: applications and examples} we demonstrate how our results significantly extend the classes of algebras for which the answer to the motivating question is known, including \cref{thm: intro3} and several explicit examples. 

\subsection*{Conventions and notation}

Throughout this paper, let $n$ denote a positive integer. We always work over an algebraically closed field $k$.

Let $\Lambda$ be an algebra. We denote by $\modu \Lambda$ the category of finitely presented right modules over $\Lambda$. If $\Lambda = \oplus_{i \geq 0} \Lambda_i$ is positively graded, we write $\gr \Lambda$ for the category of finitely presented graded right $\Lambda$-modules and degree $0$ morphisms and $\stgrmodu \Lambda$ for the associated stable category. 

The notation $D$ is used for the duality $D(-) \coloneq \Hom_k(-,k)$, and we write $\Thick(X)$ for the thick subcategory generated by an object $X$. The composition of two consecutive arrows $i \xrightarrow[]{a} j \xrightarrow[]{b} k$ in a quiver is denoted \mbox{by $ab$}.

\section{Preliminaries}\label{sec: preliminaries}

In this section we give an overview of some definitions and results that are needed in the rest of the paper. We first give a brief introduction to the (\textbf{Fg})-condition in \cref{subsec: fg-condition}, before presenting some basic results regarding the centers of a positively graded algebra in \cref{subsec: center of graded algebra}. In \cref{subsec: dg-algebra} and \cref{subsec: A_infty} we recall some necessary background concerning dg-algebras and $A_\infty$-algebras.

\subsection{The (\textbf{Fg})-condition} \label{subsec: fg-condition}

In this subsection we briefly recall notions related to the (\textbf{Fg})-condition, as well as stating and providing proofs of two results that are known to the experts, but not explicitly stated in the literature. For a more thorough introduction, see e.g.\ \cite{Solberg'06,W19}.

The \emph{enveloping algebra} of an algebra $\Lambda$ is given by \mbox{$\Lambda^{e} \coloneq \Lambda^{\op} \otimes_k \Lambda$}. Right modules over $\Lambda^{e}$ correspond to $\Lambda$-$\Lambda$-bimodules $M$ satisfying that $\lambda m = m\lambda$ for $\lambda \in k$ and $m \in M$. Note that we can regard $\Lambda$ as a right $\Lambda^{e}$-module by setting $a \cdot (a' \otimes_k a'') = a'aa''$. The $i$-\emph{th Hochschild cohomology} of $\Lambda$ can be defined as $\HH^{i}(\Lambda) \coloneq \Ext_{\Lambda^{e}}^{i}(\Lambda,\Lambda)$. Moreover, we call $\HH^*(\Lambda) = \oplus_{i \geq 0}\HH^{i}(\Lambda)$ the \emph{Hochschild cohomology ring} of $\Lambda$.

If $M$ is a right $\Lambda$-module and $\eta \in \HH^{i}(\Lambda)$ is regarded as an exact sequence of $\Lambda^{e}$-modules, then $M \otimes_\Lambda \eta$ remains exact since $\eta$ is split exact when considered as a sequence of left $\Lambda$-modules. This means that $M \otimes_\Lambda \eta$ is an element of $\Ext^{i}_\Lambda(M,M)$, and in this way one obtains a graded algebra morphism from $\HH^*(\Lambda)$ to $\Ext^*_\Lambda(M,M) = \oplus_{i \geq 0} \Ext^i_\Lambda(M,M)$ which we call the \emph{characteristic morphism}; see e.g.\ \cite[Section 3]{Solberg'06}. Note that whenever we regard $\Ext^*_\Lambda(M,M)$ as an $\HH^*(\Lambda)$-module, we mean with the module action induced by the characteristic morphism.

In the lemma below, we use the notation 
\[
\HH^*(\Lambda,-) \coloneq \Hom_{\De(\Lambda^{e})}^*(\Lambda, -) =\oplus_{i \in \mathbb{Z}}\Hom_{\De(\Lambda^{e})}(\Lambda, -[i]),
\]
where $\De(\Lambda^{e})$ is the derived category and $[i]$ denotes the $i$-th shift functor. Note that $\HH^*(\Lambda,\Lambda) \simeq \HH^*(\Lambda)$.

\begin{lem}\label{lem: noetherian endomorphism algebra and exact sequence}
Assume that $\HH^*(\Lambda)$ is noetherian. Consider a distinguished triangle $X \rightarrow Y \rightarrow Z \rightarrow X[1]$ in $\De(\Lambda^{e})$. If $~\HH^*(\Lambda,X)$ and $\HH^*(\Lambda,Z)$ are finitely generated \mbox{$\HH^*(\Lambda)$}-modules, then so is $\HH^*(\Lambda,Y)$. 
\end{lem}

\begin{proof}
The long exact sequence induced by applying $\Hom_{\De(\Lambda^{e})}(\Lambda,-)$ to the distinguished triangle $X \rightarrow Y \rightarrow Z \rightarrow X[1]$ yields an exact sequence 
\[
\HH^*(\Lambda,X) \rightarrow \HH^*(\Lambda,Y) \rightarrow \HH^*(\Lambda,Z).
\]
Since $\HH^*(\Lambda)$ is noetherian, the image of the rightmost morphism in this sequence is a finitely generated \mbox{$\HH^*(\Lambda)$-module} by our assumption that $\HH^*(\Lambda,Z)$ is finitely generated. Hence, as $\HH^*(\Lambda,X)$ is a finitely generated $\HH^*(\Lambda)$-module and thus the image of the leftmost morphism is as well, we deduce that $\HH^*(\Lambda,Y)$ is also a finitely generated $\HH^*(\Lambda)$-module.
\end{proof}

Assume now that $\Lambda$ is finite dimensional and set $S \coloneq \Lambda/\rad \Lambda$. We say that $\Lambda$ \emph{satisfies the} (\textbf{Fg})\emph{-condition} if $\HH^{*}(\Lambda)$ is noetherian and $\Ext^{*}_{\Lambda}(S,S)$ is finitely generated as an $\HH^{*}(\Lambda)$-module; see e.g.\ \cite[Proposition 5.7]{Solberg'06}. We note the following.

\begin{prop}\label{prop: equivalent fg}
Let $\Lambda$ be finite dimensional. Consider $M \in \modu \Lambda$ and assume that $\Thick (M) = \De^b(\Lambda)$. Then $\Lambda$ satisfies the \emph{(\textbf{Fg})}-condition if and only if $\HH^{*}(\Lambda)$ is noetherian and $\Ext^{*}_{\Lambda}(M,M)$ is finitely generated as an $\HH^{*}(\Lambda)$-module.
\end{prop}

\begin{proof}
This result follows by a variation of an argument from \cite[\mbox{Proposition 2.4}]{EHTSS} where we replace filtrations in simple $\Lambda^{e}$-modules by taking cones of morphisms \mbox{in $\De^b(\Lambda^{e})$}.

Note first that we clearly have 
\[
S \in \Thick (M) = \De^b(\Lambda)
\]
and that $\Hom_k(M, -)$ defines a triangulated functor from $\De^b(\Lambda)$ to $\De^b(\Lambda^{e})$. This implies that $\Hom_k (M,S) \in \Thick(\Hom_k (M,M))$. Similarly, we observe that $\Hom_k(-, S)$ is a triangulated functor from $\De^b(\Lambda)^{\op}$ to $\De^b(\Lambda^{e})$, and thus 
\begin{equation} \label{eq: proof of equivalent statements fg}
    \Hom_k (S,S) \in \Thick(\Hom_k (M,S)) \subseteq \Thick(\Hom_k (M,M)).
\end{equation}

We next use that $\De^b(\Lambda^{e}) = \Thick(\Lambda^{e}/\rad \Lambda^{e})$ since $\Lambda^{e}$ is finite dimensional. Note that $\rad \Lambda^{e} \simeq \rad \Lambda \otimes_k \Lambda +  \Lambda \otimes_k \rad \Lambda$ as $k$ is algebraically closed, and thus any simple $\Lambda^{e}$-module is of the form $D(S_i) \otimes_k S_j \simeq \Hom_k(S_i, S_j)$ for simple \mbox{$\Lambda$-modules} $S_i$ and $S_j$. This yields that $\Lambda^{e}/\rad \Lambda^{e}$ is isomorphic to $\Hom_k (S,S)$, which gives $\De^b(\Lambda^{e}) = \Thick (\Hom_k (S,S))$. Combining this with \eqref{eq: proof of equivalent statements fg}, we obtain
\[
\De^b(\Lambda^{e}) = \Thick (\Hom_k (S,S)) = \Thick(\Hom_k (M,M)).
\]

In particular, the argument above shows that $\Hom_k (S,S) \in \Thick(\Hom_k (M,M))$ and $\Hom_k (M,M) \in \Thick(\Hom_k (S,S))$. We can thus use \cref{lem: noetherian endomorphism algebra and exact sequence} to deduce that 
\[
\Ext^*_{\Lambda}(M,M) \simeq \HH^*(\Lambda, \Hom_k (M,M))
\]
is finitely generated as an $\HH^*(\Lambda)$-module if and only if the same is true for 
\[
\Ext^*_{\Lambda}(S,S) \simeq \HH^*(\Lambda, \Hom_k (S,S)). 
\]
Note that the two isomorphisms above follow from \cite[Theorem IX.2.8a]{CE56}.
\end{proof}

\subsection{The centers of a graded algebra}\label{subsec: center of graded algebra}

In this subsection we provide proofs of some results concerning the center and the graded center of a positively graded algebra that are needed later in the paper. Note that we use the notation $Z(\Lambda)$ for the (ungraded) center of an algebra $\Lambda$. If $\Lambda = \oplus_{i \geq 0} \Lambda_i$ is positively graded, then the \emph{graded center} of $\Lambda$ is given by
\[
Z_{\gr}(\Lambda) \coloneq \{x \in \Lambda_i \mid xy = (-1)^{ij}yx ~\text{for any}~ y \in \Lambda_j\}.
\]

Our first observation is that the center $Z(\Lambda)$ of a graded algebra $\Lambda = \oplus_{i \geq 0} \Lambda_i$ is again a graded algebra. We include a proof for the convenience of the reader.

\begin{prop}\label{prop: center of graded algebra is graded}
Let $\Lambda = \oplus_{i \geq 0} \Lambda_i$ be a positively graded algebra.
Then $Z(\Lambda)$ is a positively graded subalgebra of $\Lambda$. 
\end{prop}

\begin{proof}
Let $z \in Z(\Lambda)$. Since $\Lambda$ is positively graded, we can write 
\[
z = z_0 + z_1 + \cdots z_t
\]
with $z_i \in \Lambda_i$. For $\lambda \in \Lambda$, we then get
\[
\lambda z = \lambda z_0 + \lambda z_1 + \cdots \lambda z_t
\]
and
\[
z \lambda = z_0 \lambda + z_1 \lambda + \cdots z_t \lambda.
\]
Consequently, an equality $\lambda z = z\lambda$ implies that $\lambda z_i = z_i \lambda$ for all $i$, so \mbox{$z_i \in Z(\Lambda) \cap \Lambda_i$}. This gives an induced grading $Z(\Lambda)_i = Z(\Lambda) \cap \Lambda_i$, making $Z(\Lambda)$ a positively graded subalgebra of $\Lambda$.
\end{proof}

If $\Lambda = \oplus_{i \geq 0} \Lambda_i$ is a positively graded algebra and $\ell$ is a positive integer, then the $\ell$-\emph{Veronese subalgebra of $\Lambda$} is given by $\Lambda_{\ell *} \coloneq \oplus_{i \geq 0} \Lambda_{\ell i}$. Note that \cref{prop: center of graded algebra is graded} allows us to form \mbox{$\ell$-Veronese} subalgebras of the center of a positively graded algebra.

The next result enables us to pass between finite generation conditions formulated in terms of centers to conditions formulated in terms of graded centers. 

\begin{prop}\label{prop: mf over gr comm subalgebra iff mf over l-Veronese of that subalgebra}
Let $\Lambda = \oplus_{i \geq 0} \Lambda_i$ be a positively graded algebra with $Z \subseteq \Lambda$ a commutative or graded commutative subalgebra. The following statements are equivalent:
\begin{enumerate}
    \item $Z$ is noetherian and $\Lambda$ is module finite over $Z$.
    \item $Z_{2 *}$ is noetherian and $\Lambda$ is module finite over $Z_{2 *}$.
\end{enumerate}
\end{prop}

\begin{proof}
Let us first assume \emph{(2)}. Observe that $\Lambda$ is module finite over $Z$ since $Z_{2*} \subseteq Z$. To see that $Z$ is noetherian, note that $Z$ is a $Z_{2 *}$-submodule of $\Lambda$. As $\Lambda$ is module finite over $Z_{2 *}$ and $Z_{2 *}$ is noetherian, we have that $Z$ is also module finite over $Z_{2*}$. This yields that $Z$ is noetherian as a module over $Z_{2*}$. Since any chain of ideals in $Z$ can be regarded as a chain of $Z_{2*}$-submodules of $Z$, we conclude that $Z$ is a noetherian algebra, showing \emph{(1)}.

Assume now that \emph{(1)} holds. We first prove that $\Lambda$ is module finite over $Z_{2 *}$. As $\Lambda$ is module finite over $Z$ by assumption, it suffices to show that $Z$ is module finite over $Z_{2 *}$. Since $Z$ is noetherian and positively graded, we have that $Z_{>0}$ is a finitely generated $Z$-module. We can thus pick finitely many homogeneous generators of $Z_{>0}$ such that $g_{j}$ (resp.\ $h_{j}$) is the $j$-th generator of even (resp.\ odd) degree. It is straightforward to see that $Z$ is module finite over $Z_{2 *}$ provided that any $x \in Z_{2 i + 1}$ for an integer $i \geq 0$ can be written in the form 
\[
x = \sum_{j} h_{j} z_{j}(x)
\]
for homogeneous elements $z_{j}(x) \in Z_{2 *}$.
To show that this condition is indeed satisfied, observe first that $x \in Z_{2i+1}$ can be written as
\[
x = \sum_j g_{j} y_{j}(x) + \sum_j h_{j} z_{j}(x)
\]
for homogeneous elements $y_{j}(x)$ of odd degree. We now proceed by induction \mbox{on $i$}. The claim holds in the case $i = 0$, since the first sum in the expression above is then zero as $g_{j}$ is a generator of $Z_{>0}$ and thus is of positive degree. We next assume that the claim holds for any $0 \leq m < i$ and show that it then also holds for $i$. Again using that $g_{j}$ is of positive degree, we know that the degree of $y_{j}(x)$ must be less than $2i + 1$. 
Applying the induction hypothesis, we get 
\begin{align*}
g_{j}y_{j}(x) 
& = g_{j}\sum_{k}h_{k}z_{k}(y_{j}(x)). 
\end{align*}
Using that $Z$ is commutative or graded commutative, this allows us to rewrite every term in the sum $\sum_j g_{j} y_{j}(x)$ as an expression of the desired form, and so the claim follows. We can thus conclude that $\Lambda$ is module finite over $Z_{2 *}$.

To see that $Z_{2*}$ is noetherian, note that any ideal $I \subseteq Z_{2*}$ gives rise to an ideal 
\[
I + Z_{2* + 1}I = I \oplus Z_{2* + 1}I \subseteq Z,
\]
where we use the notation $Z_{2* + 1}=\oplus_{i \geq 0}Z_{2i+1}$. Note also that we have an inclusion of ideals $I \subseteq J \subseteq Z_{2*}$ if and only if 
\[
I \oplus Z_{2* + 1}I \subseteq J \oplus Z_{2* + 1}J \subseteq Z.
\]
Moreover, the analogue statement holds when replacing $\subseteq$ by $=$. For all of the observations above, we use that $Z$ is commutative or graded commutative and that $I \cap Z_{2* + 1}I = \{0\}$ since the homogeneous components of elements in $I$ and $Z_{2* + 1}I$ are non-zero in even and odd degrees, respectively. It follows that any ascending chain of ideals in $Z_{2*}$ stabilizes since the induced chain does so in $Z$, and hence $Z_{2*}$ is noetherian. 
This finishes the proof that \emph{(1)} implies \emph{(2)}.
\end{proof}

\subsection{Background on dg-algebras}\label{subsec: dg-algebra}
In this subsection we recall the definition of a dg-algebra and describe the dg-algebra structure in a case that is particularly important for the paper. The subsection is based on \cite{K94,Keller-dg-survey}, and the reader is referred there for a more thorough introduction. It should moreover be noted that any dg-algebra is an example of an $A_\infty$-algebra, a notion we review in \cref{subsec: A_infty}.

Intuitively, a dg-algebra is an algebra that is also a complex, and where the multiplication and the differential are compatible. More precisely, a \emph{dg-algebra} consists of a $\mathbb{Z}$-graded algebra $\Gamma = \oplus_{i \in \mathbb{Z}} \Gamma^{i}$ together with a graded $k$-linear map $d \colon \Gamma \to \Gamma$ called the \emph{differential}. The differential is required to be of degree $1$, meaning that $d(\Gamma^i)\subseteq \Gamma^{i+1}$ for all $i$, and should satisfy $d^2=0$. Moreover, the Leibniz rule
\[
d(fg) = d(f)g + (-1)^ifd(g)
\]
holds for homogeneous elements $f,g \in \Gamma$, where $f \in \Gamma^i$. Any $k$-algebra can be thought of as a dg-algebra concentrated in degree $0$. Furthermore, any graded algebra yields a dg-algebra if endowed with the trivial differential.

We now consider an important source of dg-algebras which plays a central role in the rest of the paper. For this, let $M \in \modu \Lambda$ for a  finite dimensional \mbox{algebra $\Lambda$}. Consider \mbox{$\Gamma \coloneq \REnd_\Lambda(M) = \RHom_\Lambda(M,M)$}, where $\RHom$ denotes the derived Hom functor. We compute $\Gamma$ with respect to a fixed projective resolution $X = pM$ of $M$, and think of $\Gamma$ as a graded algebra with 
\[
\Gamma^i = \prod_{m \in \mathbb{Z}} \Hom_{\Lambda}(X^{m}, X^{m + i})
\]
for $i \in \mathbb{Z}$. The graded algebra $\Gamma = \oplus_{i \in \mathbb{Z}} \Gamma^{i}$ then yields a dg-algebra when endowed with the standard super commutator differential defined by
\[
d(f) = d_{X} \circ f - (-1)^i f\circ d_{X}
\]
for $f \in \Gamma^i$. It should be noted that the Leibniz rule holds in this case since multiplication in $\Gamma$ is defined by composition.

\subsection{Background on $A_\infty$-algebras}\label{subsec: A_infty}

In this subsection we briefly recall necessary background on $A_\infty$-algebras that will be used in \cref{sec: general results} and \cref{sec: higher Koszul}. For a more thorough introduction to this topic, see e.g.\ \cite{Keller-Ainfty-survey-2,Keller-Ainfty-survey-3,Keller-Ainfty-survey}. 

An \emph{$A_\infty$-algebra} is a $\mathbb{Z}$-graded vector space 
\[
\Gamma = \bigoplus_{i \in \mathbb{Z}} \Gamma^{i}
\]
together with graded $k$-linear maps  
\[
m_d \colon \Gamma^{\otimes d} \rightarrow \Gamma
\]
of degree $2 - d$ for $d \geq 1$ satisfying certain relations. We will not make explicit use of these relations except in a few special cases, and refer the reader to e.g.\ \cite[Section 3.1]{Keller-Ainfty-survey-2} for their general description. It follows from these relations that $\Gamma$ is a complex with differential $m_1$. Moreover, if $m_d = 0$ for $d \geq 3$, then $\Gamma$ is a dg-algebra with multiplication given by the map $m_2 \colon \Gamma \otimes \Gamma \rightarrow \Gamma$. Conversely, any dg-algebra $\Gamma$ yields an $A_\infty$-algebra with $m_d=0$ for $d \geq 3$ by choosing $m_1$ and $m_2$ to be given by its differential and its multiplication, respectively. 

We need the following result, where we write $\Hm^*(\Gamma) = \oplus_{i \in \mathbb{Z}}\Hm^i(\Gamma)$ for the cohomology of an $A_\infty$-algebra $\Gamma$. For the definition of morphisms and quasi-isomorphisms of $A_\infty$-algebras, see e.g.\ \cite[Section 3.4]{Keller-Ainfty-survey-2}. Note that in the theorem below, the notation $m_d$ is used for the maps giving the $A_\infty$-structure of $\Hm^*(\Gamma)$, while $m_d^\Gamma$ is used for the maps associated to $\Gamma$.

\begin{thm}\label{thm: A_inf} \emph{(}\cite{Kadeishvili}\emph{, see also }\cite[Theorem 3.3]{Keller-Ainfty-survey-2}\emph{.)}
    Let $\Gamma$ be an $A_\infty$-algebra. Then $\Hm^*(\Gamma)$ admits an $A_\infty$-algebra structure such that the following statements hold:
    \begin{enumerate}
        \item One has $m_1 = 0$, and $m_2$ is induced by $m_2^\Gamma$.
        \item There is a quasi-isomorphism of $A_\infty$-algebras $\Hm^*(\Gamma) \rightarrow \Gamma$ that induces the identity in cohomology. 
    \end{enumerate}
    Moreover, this structure is unique up to (non-unique) isomorphism of $A_\infty$-algebras. 
\end{thm}

The theorem above will be particularly relevant in the case of the dg-algebra \mbox{$\Gamma = \REnd_\Lambda(M)$} for a $\Lambda$-module $M$,  allowing us to endow \mbox{$\Hm^*(\Gamma) \simeq \Ext^*_\Lambda(M,M)$} with an $A_\infty$-structure with $m_1 = 0$ and $m_2$ the usual multiplication satisfying that $\Gamma$ and $\Hm^*(\Gamma)$ are quasi-isomorphic as $A_\infty$-algebras. Note that an $A_\infty$-algebra with $m_1 = 0$ is said to be \emph{minimal}. 

In the setup of \cref{thm: A_inf}, if the $A_\infty$-algebra structure of $\Hm^*(\Gamma)$ can be chosen such that $m_d = 0$ for $d \geq 3$ (i.e.\ it can be chosen to simply be an associative graded algebra), then $\Gamma$ is called \emph{formal}.

\section{Some general tools} \label{sec: general results}

The aim of this section is to establish \cref{prop: general result}, which is a key ingredient in the proof of \cref{thm: intro1}. Although the focus of this paper is to investigate the (\textbf{Fg})-condition from the viewpoint of higher Koszul algebras, the results of this section are applicable in a more general setup. Note that in \cref{sec: higher Koszul}, we will specialize to the case where $\Lambda$ is an $n$-$T$-Koszul algebra and $M=T$.

\begin{setup}
Throughout this section, let $\Lambda$ be a finite dimensional algebra and consider $M \in \modu \Lambda$. Let $X = pM$ be a fixed projective resolution of $M$, and set \mbox{$\Gamma \coloneq \REnd_\Lambda(M)$}.
\end{setup}

Recall that we think of $\Gamma$ as a dg-algebra as described in \cref{subsec: dg-algebra}. The projective resolution $X = pM$ of $M$ is an $\Gamma$-$\Lambda$-dg-bimodule in the sense of \cite[Section 3.8]{Keller-dg-survey}. 

Using the theory of standard lifts as in \cite[Section 7.3]{K94}, we get an equivalence
\[
\RHom_\Lambda(X, -) \colon \Thick (M) \longrightarrow \De^{\perf}(\Gamma),  
\]
where $\De^{\perf}(\Gamma)$ denotes the subcategory of perfect objects in the derived category $\De(\Gamma)$. Note that we here use that $\De(\Lambda)$ is idempotent complete. The equivalence above has quasi-inverse given by 
\[
- \otimes_{\Gamma}^{\bbL} X \colon \De^{\perf}(\Gamma) \longrightarrow \Thick (M).
\]
This yields that the functor 
\[
- \otimes_{\Gamma}^{\bbL} X \colon \De^{\perf}(\Gamma) \longrightarrow \De(\Lambda)
\]
is fully faithful.

We next want to prove that the functor
\[
X \otimes_{\Lambda}^{\bbL} - \colon \De^{\perf} (\Lambda^{\op}) \longrightarrow \De(\Gamma^{\op})
\]
is also fully faithful whenever $\Thick (M) = \De^b(\Lambda)$. 
This is shown in \cite[Theorem 4.6 b)]{K03} in the case where
\[
- \otimes_{\Gamma}^{\bbL} X \colon \De(\Gamma) \longrightarrow \De(\Lambda)
\]
is an equivalence. An analogue proof works under our assumptions, as demonstrated in the following.

\begin{lem}\label{lem: tensor ff 2}
If $\Thick (M) = \De^b(\Lambda)$, then the functor
\[
X \otimes_{\Lambda}^{\bbL} - \colon \De^{\perf} (\Lambda^{\op}) \longrightarrow \De(\Gamma^{\op})
\]
is fully faithful.
\end{lem}

\begin{proof}
Note that the transposition functor 
\[
\Tr_\Lambda(-) \coloneq \RHom_\Lambda(-,\Lambda) \colon \De(\Lambda) \rightarrow \De(\Lambda^{\op})^{\op}
\]
induces an equivalence
\[
\De^{\perf} (\Lambda) \rightarrow \De^{\perf}(\Lambda^{\op})^{\op}
\]
since $\Hom_\Lambda(-, \Lambda)$ yields an equivalence between $\proj \Lambda$ and $(\proj \Lambda^{\op})^{\op}$; see e.g.\ the middle of page 107 in \cite{ASS06}. For $Q \in \De^{\perf} (\Lambda)$, we thus have natural isomorphisms 
\begin{align*}
X \otimes_{\Lambda}^{\bbL} \RHom_\Lambda(Q,\Lambda) 
& \xrightarrow{\sim} \RHom_\Lambda(Q,X) \\
& \xrightarrow{\sim} \RHom_\Gamma(\RHom_\Lambda(X,Q),\RHom_\Lambda(X,X)) \\
& \xrightarrow{\sim} \RHom_\Gamma(\RHom_\Lambda(X,Q), \Gamma). 
\end{align*}
To get these isomorphisms, we use for the first that $Q \in \De^{\perf} (\Lambda)$, for the second that 
\[
\RHom_\Lambda(X,-) \colon \Thick (M) = \De^b(\Lambda) \longrightarrow \De^{\perf}(\Gamma)
\]
is fully faithful, and for the third that $\RHom_\Lambda(X,-)$ sends $X$ to $\Gamma$. This yields that we have a natural isomorphism 
\[
(X \otimes_{\Lambda}^{\bbL} -) \circ \Tr_\Lambda \xrightarrow{\sim} \Tr_\Gamma \circ \RHom_\Lambda(X,-) 
\]
of functors $\De^{\perf} (\Lambda) \rightarrow \De(\Gamma^{\op})^{\op}$. Hence,
\[
X \otimes_{\Lambda}^{\bbL} - \colon \De^{\perf} (\Lambda^{\op}) \longrightarrow \De(\Gamma^{\op})
\]
is fully faithful since
\[
\Tr_\Gamma \circ \RHom_\Lambda(X,-) \circ \Tr_{\Lambda}^{-1} \colon \De^{\perf} (\Lambda^{\op})^{\op} \rightarrow \De(\Gamma^{\op})^{\op}
\]
is fully faithful, where we write $\Tr_{\Lambda}^{-1}$ for a quasi-inverse of $\Tr_{\Lambda}$.
\end{proof}

Let $R$ and $S$ be dg-algebras. Following \cite[Section 3.1]{Briggs-Gelinas}, an \mbox{$R$-$S$-dg-bimodule $N$} is called \emph{homologically balanced} if the natural morphisms $R \rightarrow \REnd_S(N)$ and \mbox{$S^{\op} \rightarrow \REnd_{R^{\op}}(N)$} are both quasi-isomorphisms. 

The following result is needed in order to prove \cref{prop: general result}.

\begin{prop}\label{prop: homologically balanced}
If $\Thick (M) = \De^b(\Lambda)$, then the $\Gamma$-$\Lambda$-dg-bimodule $X = pM$ is homologically balanced.
\end{prop}

\begin{proof}
As $\Gamma = \REnd_\Lambda(M)$, the first morphism in the definition of a homologically balanced $\Gamma$-$\Lambda$-dg-bimodule is trivially a quasi-isomorphism. It thus remains to show that the natural morphism
\[
\Lambda^{\op} \rightarrow \REnd_{\Gamma^{\op}}(X)
\]
is a quasi-isomorphism. This follows from \cite[Lemma 4.2]{K94}, see also \cite[\mbox{Section 3.2}]{K03}, since the functor 
\[
X \otimes_{\Lambda}^{\bbL} - \colon \De^{\perf} (\Lambda^{\op}) \longrightarrow \De(\Gamma^{\op})
\]
is fully faithful by \cref{lem: tensor ff 2}.
\end{proof}

Recall that we can endow $\Hm^*(\Gamma) \simeq \Ext^*_\Lambda(M,M)$ with the structure of a minimal $A_\infty$-algebra satisfying the conditions in \cref{thm: A_inf} as described at the end of \cref{subsec: A_infty}. In the proof of \cref{prop: general result} below, we apply a result from \cite{Briggs-Gelinas} concerning the \emph{$A_\infty$-center} of $\Ext^*_\Lambda(M,M)$.
For the definition of the $A_\infty$-center of a minimal $A_\infty$-algebra, see \cite[\mbox{Definition 3.7}]{Briggs-Gelinas}.

We are now ready to prove \cref{prop: general result}.

\begin{prop} \label{prop: general result}
Let $\Gamma = \REnd_\Lambda(M)$ be formal and assume \mbox{$\Thick (M) = \De^b(\Lambda)$}. Then $\Lambda$ satisfies the \emph{(\textbf{Fg})}-condition if and only if $Z_{\gr}(\Ext^{*}_{\Lambda}(M,M))$ is noetherian and $\Ext^{*}_{\Lambda}(M,M)$ is module finite over $Z_{\gr}(\Ext^{*}_{\Lambda}(M,M))$.
\end{prop}

\begin{proof} 
Since $\Thick (M) = \De^b(\Lambda)$, we know from \cref{prop: homologically balanced} that the $\Gamma$-$\Lambda$-dg-bimodule $X = pM$ is homologically balanced. Hence, the characteristic morphism 
\[
\HH^{*}(\Lambda) \rightarrow \Ext_\Lambda^{*}(M,M)
\]
surjects onto the $A_\infty$-center of $\Ext_\Lambda^{*}(M,M) \simeq \Hm^*(\Gamma)$ by \cite[Corollary 3.9]{Briggs-Gelinas}. Since $\Gamma$ is formal, the $A_{\infty}$-center coincides with the graded center $Z_{\gr}(\Ext_\Lambda^{*}(M,M))$, as noted e.g.\ on the top of page 29 of \cite{Briggs-Gelinas}. Using this together with \cref{prop: equivalent fg} and \cref{prop: mf over gr comm subalgebra iff mf over l-Veronese of that subalgebra}, the claim now follows by analogue arguments as those used to show \cite[Theorem 1.3]{Erdmann-Solberg-1}.
\end{proof}

\section{The (\textbf{Fg})-condition for higher Koszul algebras} \label{sec: higher Koszul}

In this section we investigate when an $n$-$T$-Koszul algebra satisfies the (\textbf{Fg})-condition and connect this to the theory of higher representation infinite algebras. Throughout the rest of the paper, we always let $\Lambda = \oplus_{i \geq 0} \Lambda_i$ be positively graded, where $\Lambda_0$ is a finite dimensional basic algebra. We assume that $\Lambda$ is locally finite dimensional, meaning that $\Lambda_i$ is finite dimensional as a vector space for each $i \geq 0$. Note that $\Lambda_0$ is assumed to be basic for consistency with \cite{HS,Madsen}; see \cref{rem: basic}.

We start by recalling relevant definitions related to $n$-$T$-Koszul algebras in \cref{subsec: background n-T}, before showing that the dual of an $n$-$T$-Koszul algebra is the cohomology of a formal $A_\infty$-algebra in \cref{subsec: formality}. In \cref{subsec: n-T and fg} we combine this with the results in \cref{sec: general results} to characterize when an $n$-$T$-Koszul algebra satisfies the \mbox{(\textbf{Fg})-condition} and prove \cref{thm: intro1} from the introduction. We next specialize to the case of a graded symmetric $n$-$T$-Koszul algebra of highest degree $1$, where we establish a strong connection between the (\textbf{Fg})-condition and higher representation infinite algebras, leading to \cref{thm: intro2}.

\subsection{Background on $n$-$T$-Koszul algebras} \label{subsec: background n-T}

In this subsection we provide an overview of definitions from \cite{HS} that are used in the rest of the paper. Recall that $n$ denotes a positive integer, and note that the definitions presented here recover notions from \cite{Madsen} in the case $n=1$. 

We first recall what it means for a module to be graded $n\mathbb{Z}$-orthogonal. A non-zero graded $\Lambda$-module $M = \oplus_{i\in\mathbb{Z}}M_i$ is \emph{concentrated in degree $0$} if \mbox{$M_i  = 0$} for \mbox{$i \neq 0$.} The \textit{$j$-th graded shift} of a graded $\Lambda$-module $M = \oplus_{i\in\mathbb{Z}}M_i$ is given by the graded module $M\langle j \rangle$ with \mbox{$M\langle j \rangle_i = M_{i-j}$}.

\begin{defin}
	Let $T$ be a finitely generated basic graded $\Lambda$-module concentrated in degree $0$. We say that $T$ is \emph{graded $n\mathbb{Z}$-orthogonal} if 
	\[\Ext^{i}_{\gr \Lambda}(T,T\langle j \rangle) = 0 \]
	for $i \neq nj$.
\end{defin}

\begin{remark}\label{rem: basic}
    A graded $n \mathbb{Z}$-orthogonal module is assumed to be basic for consistency with \cite{HS, Madsen}. We note that the proofs of the results in \cref{sec: higher Koszul} do not rely on this assumption; see \cite[Remark 3.6]{HS}.
\end{remark}

We are now ready to define higher Koszul algebras, or $n$-$T$-Koszul algebras. For the definition of a tilting module, see e.g.\ \cite{M86}. 

\begin{defin}
	Assume $\gldim \Lambda_0 < \infty$ and let $T$ be a graded $\Lambda$-module concentrated in degree $0$. We say that $\Lambda$ is \emph{$n$-$T$-Koszul} or \emph{$n$-Koszul with respect to $T$} if the following conditions hold:
	\begin{enumerate}
		\item $T$ is a tilting $\Lambda_0$-module.
		\item $T$ is graded $n\mathbb{Z}$-orthogonal as a $\Lambda$-module.
	\end{enumerate}
\end{defin}

Given an $n$-$T$-Koszul algebra, we can associate a version of the Koszul dual. This $n$-$T$-Koszul dual plays a crucial role in the rest of the paper.

\begin{defin}\label{def: n-T-Koszul dual}
Let $\Lambda$ be an $n$-$T$-Koszul algebra. The \emph{$n$-$T$-Koszul dual of $\Lambda$} is given by $\Lambda^! \coloneq  \oplus_{i \geq 0} \Ext_{\gr \Lambda}^{ni}(T,T\langle i \rangle)$.
\end{defin}

Note that the definition of the $n$-$T$-Koszul dual $\Lambda^!$ of an $n$-$T$-Koszul algebra $\Lambda$ relies on the integer $n$ and the module $T$. Since it is possible for an algebra to be both $n$-$T$-Koszul and $n'$-$T'$-Koszul for $n \neq n'$ and $T \not\simeq T'$, the notation for the $n$-$T$-Koszul dual is potentially ambiguous. However, it will for us always be clear from context which $n$-$T$-Koszul structure the dual is computed with respect to, and it is hence not necessary to include the data $n$ and $T$ in the notation for the Koszul dual.

\subsection{$n$-$T$-Koszul algebras and formality}\label{subsec: formality}

In this subsection we show that the \mbox{$n$-$T$-Koszul} dual of an $n$-$T$-Koszul algebra $\Lambda$ is the cohomology of a formal $A_\infty$-algebra. More precisely, we demonstrate that $\Lambda^!$ is isomorphic to the cohomology of $\Gamma = \REnd_{\Lambda}(T)$ in \cref{prop: Koszul dual as Ext}, before showing that $\Gamma$ is a formal \mbox{$A_\infty$-algebra} in \cref{thm: A_infty-structure on Koszul dual}. Note that we think of $\Lambda^!$ as a graded algebra given by putting $\Ext_{\gr \Lambda}^{ni}(T,T\langle i \rangle)$ in degree $ni$ and having zero in all degrees not divisible by $n$. Since $\Lambda$ is a graded algebra, there is also induced a second grading on $\Lambda^!$ which is often referred to as an \emph{internal grading} or \emph{Adams grading}. Note that this internal grading is compatible with the cohomological grading in the sense that $\Lambda^!$ is bigraded, i.e.\ graded over $\mathbb{Z}\times\mathbb{Z}$. The internal grading will play a key role in the proof of \cref{thm: A_infty-structure on Koszul dual}.

\begin{prop}\label{prop: Koszul dual as Ext}
    Let $T$ be a graded $n\mathbb{Z}$-orthogonal $\Lambda$-module and set $\Gamma = \REnd_{\Lambda}(T)$. The following statements hold:
    \begin{enumerate}
        \item The grading on $\Lambda$ induces an internal grading on $\Hm^*(\Gamma)$ which is compatible with the cohomological grading and is given by $\Hm^{*}(\Gamma)_j \simeq \oplus_{i \in \mathbb{Z}} \Hm^{i}(\Gamma)_j$ with $\Hm^{i}(\Gamma)_j \simeq \Ext_{\gr \Lambda}^{i}(T,T\langle j \rangle) = \Ext_{\gr \Lambda}^{nj}(T,T\langle j \rangle)$ for $i=nj$ and $0$ otherwise.
        \item If $\Lambda$ is $n$-$T$-Koszul, then we have isomorphisms of graded algebras
        \[
        \Lambda^! \simeq \Ext_{\Lambda}^{*}(T,T) \simeq \Hm^*(\Gamma).
        \]
    \end{enumerate}
\end{prop}

\begin{proof}
As demonstrated in the proof of \cite[Proposition 3.1.2]{Madsen}, we have
\[
\Ext^i_\Lambda(T,T) \simeq \prod_{j \in \mathbb{Z}}\Ext^i_{\gr \Lambda}(T,T \langle j \rangle)  
\]
for all $i \geq 0$. As $T$ is graded $n\mathbb{Z}$-orthogonal, this product equals $\Ext_{\gr \Lambda}^{nj}(T,T\langle j \rangle)$ for $i=nj$ and is $0$ if $i$ is not divisible by $n$. The claims in \emph{(1)} now follow by noting that $\Ext_{\Lambda}^{*}(T,T) \simeq \Hm^*(\Gamma)$ as the augmentation map $pT \rightarrow T$ is a quasi-isomorphism. Since we have
\[
\Ext_{\Lambda}^{*}(T,T) \simeq \bigoplus_{i \geq 0} \Ext_{\gr \Lambda}^{ni}(T,T\langle i \rangle), 
\]
part \emph{(2)} is deduced by combining the arguments above with the definition of the $n$-$T$-Koszul dual $\Lambda^!$.
\end{proof}

Recall from \cref{thm: A_inf} that we can endow the $n$-$T$-Koszul dual 
\[
\Lambda^! \simeq \Ext_{\Lambda}^{*}(T,T) \simeq \Hm^*(\Gamma)
\] 
with an $A_\infty$-structure in such a way that $m_1=0$, $m_2$ is the usual multiplication and $\Gamma = \REnd_{\Lambda}(T)$ and $\Hm^*(\Gamma)$ are quasi-isomorphic as $A_\infty$-algebras. Our next result demonstrates that this $A_\infty$-structure can be chosen such that $m_d = 0$ for $d \geq 3$. To see this, we employ a small variation on a standard trick of using an internal grading to show that an $A_\infty$-algebra is formal.

\begin{thm} \label{thm: A_infty-structure on Koszul dual}
Let $T$ be a graded $n\mathbb{Z}$-orthogonal $\Lambda$-module. 
Then $\Gamma = \REnd_\Lambda(T)$ is a formal $A_{\infty}$-algebra. In particular, this holds if $\Lambda$ is an $n$-$T$-Koszul algebra.
\end{thm}

\begin{proof}
We begin by noting that the $A_\infty$-structure on $\Hm^*(\Gamma)$ as in \cref{thm: A_inf} can be chosen such that the maps $m_d$ are homogeneous of \mbox{degree $0$} with respect to the internal grading of $\Hm^*(\Gamma)$ described in \cref{prop: Koszul dual as Ext} \emph{(1)}. In other words, the maps $m_d$ can be chosen to be homogeneous of bidegree $(2 - d, 0)$, where the first coordinate indicates the cohomological grading and the second the internal grading. To check this, one could consult the construction of the structure in \cref{thm: A_inf}, e.g.\ in \cite{Merkulov}. See also \cite[Section 2]{Ainfty-structure-on-ext-algebras} for an explicit proof in the case where $\Gamma$ is a dg-algebra, in particular the remark after \cite[Proposition 2.3]{Ainfty-structure-on-ext-algebras}. 

By part \emph{(1)} of \cref{prop: Koszul dual as Ext}, we know that $\Hm^{i}(\Gamma)_j \simeq \Ext_{\gr \Lambda}^{i}(T,T\langle j \rangle)$, which equals $\Ext_{\gr \Lambda}^{nj}(T,T\langle j \rangle)$ if $i=nj$ and is $0$ when $i$ is not divisible by $n$. Consider now a non-zero element $a_1 \otimes a_2 \otimes \cdots \otimes a_d \in \Hm^*(\Gamma)^{\otimes d}$ that is homogeneous in each grading, and let the internal degree of $a_i$ be denoted by $\lvert a_i \rvert$. As this element is of bidegree $(n\Sigma_{i} \lvert a_i \rvert, \Sigma_{i} \lvert a_i \rvert)$, it follows that the bidegree of $m_d(a_1 \otimes a_2 \otimes \cdots \otimes a_d)$ is $(n\Sigma_{i} \lvert a_i \rvert + 2 - d, \Sigma_{i} \lvert a_i \rvert)$. This implies that $m_d(a_1 \otimes a_2 \otimes \cdots \otimes a_d) = 0$ unless $d = 2$, since otherwise its cohomological degree does not equal $n$ times its internal degree, and we can conclude that $\Gamma$ is formal. 
\end{proof}

\subsection{$n$-$T$-Koszul algebras and the (\textbf{Fg})-condition}\label{subsec: n-T and fg}

We are now ready to prove \cref{thm: intro1} from the introduction, generalizing \cite[Theorem 1.3]{Erdmann-Solberg-1} to the significantly bigger class of $n$-$T$-Koszul algebras. 

\begin{thm}\label{thm: characterisation for n-T}
Let $\Lambda$ be a finite dimensional $n$-$T$-Koszul algebra. Then $\Lambda$ satisfies the \emph{(\textbf{Fg})}-condition if and only if $Z_{\gr}(\Lambda^!)$ is noetherian and $\Lambda^!$ is module finite over $Z_{\gr}(\Lambda^!)$.
\end{thm}

\begin{proof}
By \cref{thm: A_infty-structure on Koszul dual}, we know that $\Gamma = \REnd_\Lambda(T)$ is a formal $A_{\infty}$-algebra. Note that 
\[
\Thick(T) = \Thick(\Lambda_0)= \Thick(\Lambda/\rad \Lambda) = \De^b(\Lambda),
\]
where the first equality follows from $T$ being a tilting $\Lambda_0$-module. For the second equality, one uses that $\Lambda_0$ has finite global dimension, while the third holds since $\Lambda$ is finite dimensional. The desired conclusion now follows from \cref{prop: general result}, as $\Lambda^! \simeq \Ext_{\Lambda}^{*}(T,T)$ by part \emph{(2)} of \cref{prop: Koszul dual as Ext}.
\end{proof}

Our next aim is to employ the theory from \cite{HS} to establish a connection between the (\textbf{Fg})-condition and the theory of higher representation infinite algebras as introduced in \cite{HIO14}. For this, it is useful to restrict our attention to graded symmetric $n$-$T$-Koszul algebras of highest degree $1$. In particular, we will characterize when such an \mbox{$n$-$T$-Koszul} algebra $\Lambda$ satisfies the (\textbf{Fg})-condition in terms of the endomorphism algebra \mbox{$B \coloneq \End_{\gr \Lambda}(T)$} being $n$-representation tame. This is done in \cref{thm: Fg iff n-rep tame}. Recall that a positively graded algebra $\Lambda = \oplus_{i \geq 0} \Lambda_i$ of highest degree $a$ is called \emph{graded symmetric} if $\Lambda \langle -a \rangle \simeq D \Lambda$ as graded $\Lambda$-bimodules. Note in particular that any graded symmetric algebra is self-injective.

\begin{remark}
If one wants to consider our theory for a graded symmetric \mbox{algebra $\Lambda$} of highest degree $a \geq 1$, then one can look at the $a$-th quasi-Veronese $\Lambda^{[a]}$ of $\Lambda$ as \mbox{in \cite{MM11}}. Note that $\Lambda^{[a]}$ can also be defined as the covering (or smash product) of $\Lambda$ induced by the $\mathbb{Z}/a\mathbb{Z}$-grading that $\Lambda$ necessarily has by virtue of being positively graded of highest \mbox{degree $a$}; see e.g.\ \cite{BG,CM06}. By \cite[Theorem 4.1]{B23} or \cite[\mbox{Proposition 2.1}]{Sandoey24}, we know that the (\textbf{Fg})-condition holds for $\Lambda$ if and only if it does for $\Lambda^{[a]}$, provided that the characteristic of the field $k$ satisfies a reasonable condition. Since $\Lambda^{[a]}$ is graded symmetric of highest degree $1$, little is hence lost by restricting to this case. 
\end{remark}

We start by recalling necessary terminology related to higher representation infinite algebras. Let $A$ denote a finite dimensional algebra. Recall that if $A$ has finite global dimension, then $\De^b(A)$ has a Serre functor given by the derived Nakayama functor \mbox{$\nu(-) \coloneq - \otimes^{\mathbf{L}}_{A} DA$}. Using the notation \mbox{$\nu_n \coloneq \nu(-)[-n]$}, the algebra $A$ is called \mbox{\emph{$n$-representation infinite}} if $\gldim A \leq n$ and \mbox{$\Hm^{i}(\nu^{-j}_n(A)) = 0$} for \mbox{$i \neq 0$} and \mbox{$j \geq 0$} \cite[Definition 2.7]{HIO14}. 

Given an $n$-representation infinite algebra $A$, we let the \emph{$(n+1)$-preprojective algebra of $A$} be denoted by $\Pi_{n+1}A$. Recall from \cite[Lemma 2.13]{IO13} that
\[
\Pi_{n+1}A \simeq \bigoplus_{i \geq 0}\Hom_{\De^b(A)}(A, \nu^{-i}_{n}(A)).
\]
An $n$-representation infinite algebra $A$ is called $n$-\emph{representation tame} if $\Pi_{n+1}A$ is a noetherian algebra over its center, i.e.\ if the center $Z \coloneq Z(\Pi_{n+1}A)$ is noetherian and $\Pi_{n+1}A$ is module finite \mbox{over $Z$} \cite[Definition 6.10]{HIO14}. We have that $Z$ is a graded algebra by \cref{prop: center of graded algebra is graded}, which in particular allows us to consider \mbox{$\ell$-Veronese} subalgebras of $Z$.

Note that the notion of a $1$-representation infinite algebra coincides with the classical notion of a representation infinite hereditary algebra. As one might expect, such an algebra is $1$-representation tame if and only if it is tame in the classical sense. One direction of this statement is pointed out in \cite[\mbox{Example 6.11 (a)}]{HIO14} in the case where the field $k$ is assumed to be of characteristic zero. As we want to work with a field of arbitrary characteristic, we need the following result. 

\begin{prop}\label{prop: tame iff 1-tame}
    Let $A$ be a representation infinite hereditary algebra. Then $A$ is tame if and only if it is $1$-representation tame.
\end{prop}

It should be noted that the content of \cref{prop: tame iff 1-tame} is well-known to experts, where parts of the proof can be deduced from e.g.\ \cite{Baer85} or \cite{BGL87}. We include a complete proof of the statement for the convenience of the reader.

\begin{proof}
    Assume that $A$ is not tame, meaning that it is of wild representation type by the tame-wild dichotomy \cite{Drozd}. The center of the associated preprojective algebra $\Pi_2 A$ is then isomorphic to $k$ by \cite[Theorem 8.4.1 (ii)]{CBEG07}, where we note that this result holds regardless of the characteristic of the field; see \cite[Theorem \mbox{10.1.1 (ii)}]{Schedler}. As $\Pi_2 A$ is infinite dimensional, it cannot be module finite over its center, which yields that $A$ is not $1$-representation tame.

    For the reverse direction, we assume that $A$ is tame and note that the proofs of the main result of \cite{Erdmann-Solberg-1} in the cases $\widetilde{\mathbb{A}}_n, \widetilde{\mathbb{D}}_n, \widetilde{\mathbb{E}}_6, \widetilde{\mathbb{E}}_7$ and $\widetilde{\mathbb{E}}_8$ together imply that $\Pi_2 A$ is module finite over its graded center which is also noetherian. In particular, one can check that $E(\Lambda)^{\op} \simeq \Pi_2 A$ for $\Lambda$ and $E(\Lambda)$ as in \cite{Erdmann-Solberg-1}, provided that their parameters $q_i$ are chosen appropriately. Hence, the algebra $A$ is $1$-representation tame by \cref{prop: mf over gr comm subalgebra iff mf over l-Veronese of that subalgebra}.
\end{proof}

We are now ready to relate the (\textbf{Fg})-condition to the theory of higher representation infinite algebras.

\begin{thm}\label{thm: Fg iff n-rep tame}
    Let $\Lambda$ be a graded symmetric finite dimensional $(n+1)$-$T$-Koszul algebra of highest degree $1$. Then $\Lambda$ satisfies the \emph{(\textbf{Fg})}-condition if and only if $B = \End_{\gr \Lambda}(T)$ is $n$-representation tame.
\end{thm}

\begin{proof}
As $B \simeq \End_{\stgrmodu \Lambda}(T)$ by \cite[Lemma 2.5 (4)]{HS}, we have that $B$ is $n$-representation infinite by \cite[Theorem 5.2]{HS}. By \cite[Proposition 5.11]{HS}, we have $\Lambda^! \simeq \Pi_{n+1} B$ as graded algebras since $\Lambda$ is graded symmetric of highest degree $1$. By \cref{prop: center of graded algebra is graded}, we moreover know that $Z(\Pi_{n+1} B)$ is a positively graded algebra. Observe next that the center and the graded center of a graded algebra have equal $2$-Veronese subalgebras, which in particular yields $Z(\Pi_{n+1} B)_{2*} = Z_{\gr}(\Pi_{n+1} B)_{2*}$. \cref{prop: mf over gr comm subalgebra iff mf over l-Veronese of that subalgebra} thus implies that $\Pi_{n+1} B$ is module finite over its graded center that is also noetherian if and only if $B$ is $n$-representation tame. The conclusion now follows by applying \cref{thm: characterisation for n-T}.
\end{proof}

We are now ready to prove \cref{thm: intro2} from the introduction. Using a characterization result from \cite{HS}, this is an immediate consequence of the result above.

\begin{cor}\label{cor: fg}
Let $\Lambda$ be a graded symmetric finite dimensional algebra of highest degree $1$ with $\Lambda_0$ an $n$-representation infinite algebra. Then $\Lambda$ satisfies the \emph{(\textbf{Fg})}-condition if and only if $\Lambda_0$ is $n$-representation tame. 
\end{cor}

\begin{proof}
By \cite[Corollary 5.7]{HS}, the assumptions imply that $\Lambda$ is $(n+1)$-Koszul with respect to $T = \Lambda_{0}$. Using that $\End_{\gr \Lambda}(\Lambda_0) \simeq \Lambda_0$, the conclusion now follows by applying \cref{thm: Fg iff n-rep tame}.
\end{proof}

\section{Applications and examples}\label{subsec: applications and examples}
In this section we give an overview of some applications and examples demonstrating how our results significantly extend the classes of algebras for which the answer to the motivating question from the introduction is known. We note that in the examples we present, we are not aware of any other methods for verifying the (\textbf{Fg})-condition except those introduced in this paper.

Recall first that the \emph{trivial extension} of a finite dimensional algebra $A$ is given by \mbox{$\Delta A \coloneq A \oplus DA$}, with multiplication $$(a, f) \cdot (b, g) = (ab, ag + fb)$$ for \mbox{$a,b \in A$} and \mbox{$f,g \in DA$}. The trivial extension $\Delta A$ is a graded symmetric algebra, where $A$ is considered to be in degree $0$ and $DA$ to be in degree $1$.

Combining \cref{cor: fg} with \cref{prop: tame iff 1-tame} yields the following as an immediate consequence.

\begin{cor}\label{cor: known except special cases}
Let $\Lambda$ be a graded symmetric finite dimensional algebra of highest degree $1$ with $\Lambda_0$ a representation infinite hereditary algebra. Then $\Lambda$ satisfies the \emph{(\textbf{Fg})}-condition if and only if $\Lambda_0$ is tame. 
\end{cor}

The corollary above entails that the trivial extension of a representation infinite hereditary algebra $A = kQ$ satisfies the (\textbf{Fg})-condition if and only if $A$ is tame. This was known to the experts in the case where the APR-tilting class of $A$ contains a hereditary algebra whose quiver is a bipartite orientation of $Q$. In particular, it was known that trivial extensions of tame representation infinite hereditary algebras satisfy the (\textbf{Fg})-condition, except in the case of certain orientations \mbox{of $\widetilde{A}_n$}. The argument in question uses the invariance of the (\textbf{Fg})-condition under derived equivalence \cite{Kulshammer-Psaroudakis-Skartsaeterhagen}, that derived equivalence of a pair of algebras implies that of their trivial extensions \cite{Rickard-stable-paper}, and that the main result from \cite{Erdmann-Solberg-1} can be applied to trivial extensions of bipartite hereditary algebras.

Note that our approach gives a unified proof for all cases, including those in which the hereditary algebra is not derived equivalent to one which is bipartite. Such a case is illustrated in the example below.
 
\begin{example}\label{ex: case not known}
Let $A = kQ$ be the path algebra of the quiver
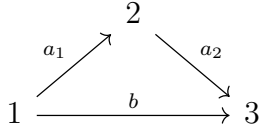
\[
\begin{tikzcd}
& 2 \drar["a_2"] &\\
1 \urar["a_1"] \arrow[rr, "b"] & & 3
\end{tikzcd}
\]
Note that even though $A$ is a tame representation infinite hereditary algebra, the arguments sketched above do not apply as no orientation of the quiver $Q$ is bipartite. The trivial extension $\Delta A$ has quiver 
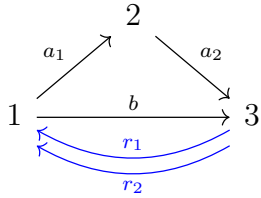
\[
\begin{tikzcd}
& 2 \drar["a_2"]  &\\
1 \urar["a_1"] \arrow[rr, "b"] & & 3 \arrow[ll, bend left, blue, "r_1" above]
\arrow[ll, bend left, blue, shift left = 1.15ex, "r_2"]
\end{tikzcd}
\]
and relations given by $br_1 - a_1a_2r_2, r_1b - r_2a_1a_2, a_2r_1, r_1a_1, br_2$ and $r_2b$. Since it has non-quadratic relations, the algebra is not Koszul in the classical sense by \mbox{\cite[Proposition 1.2.3]{Beilinson-Ginzburg-Soergel}}, and thus the results in \cite{Erdmann-Solberg-1} cannot be applied to $\Delta A$. However, \cref{cor: known except special cases} allows us to deduce immediately that $\Delta A$ indeed satisfies the (\textbf{Fg})-condition.
\end{example}

We now elaborate a bit on one interesting and useful feature of working with our theory that is exhibited in the preceding example. Namely, even if the higher Koszul algebra $\Lambda$ that we consider is not itself classically Koszul, it may still be the case that the higher Koszul dual is. This is particularly useful in the setup where $\Lambda = \Delta A$ for an $n$-representation infinite algebra $A$, since then the \mbox{$(n+1)$-$A$-Koszul} dual is the $(n+1)$-preprojective of $A$ \cite[Proposition 5.11]{HS}. As $\Pi_{n+1} A$ is known to be classically Koszul whenever $A$ is by \cite[Theorem B (b)]{GI20}, one often has access to useful formulas for computing its quiver with relations; see \cite[\mbox{Theorem C}]{Thi20} and \cite[Theorem A]{GI20}. This makes it easier to compute the (graded) center of $\Pi_{n+1}A$ and to check if $\Pi_{n+1}A$ is module finite over it, which again allows us to determine whether or not $\Delta A$ satisfies the (\textbf{Fg})-condition by \cref{cor: fg}. Note that this approach is particularly powerful when the $(n+1)$-$A$-Koszul dual has a quadratic Gr\"obner basis, as one can then expect the computations involved to be particularly tractable. Indeed, having a quadratic Gr\"obner basis is a large part of what enables the rather straightforward computations in \cite{Erdmann-Solberg-1} that we refer to in the proof of \cref{prop: tame iff 1-tame} and which are necessary for \cref{cor: known except special cases}. 

\cref{ex: tensor product} illustrates the approach sketched above for an important class of examples.

\begin{example}\label{ex: tensor product}
Let $Q$ be a non-Dynkin quiver. As $kQ$ is $1$-representation infinite, the algebra $A \coloneq kQ \otimes_k kQ$ is $2$-representation infinite by \cite[Theorem 2.10]{HIO14}. Moreover, note that $A$ is Koszul since $kQ$ is Koszul and tensor products of Koszul algebras are again Koszul.

If $Q$ is not bipartite, then the trivial extension $\Delta A$ cannot be Koszul in the classical sense as it will have non-quadratic relations \cite[Proposition 1.2.3]{Beilinson-Ginzburg-Soergel}. It is thus reasonable to expect that checking the (\textbf{Fg})-condition for $\Delta A$ directly could be difficult. Moreover, if $kQ$ is not derived equivalent to a bipartite hereditary algebra, we are not aware of simplifications involving derived equivalences, as we do not know of a candidate Koszul algebra derived equivalent to $\Delta A$.

However, regardless of whether $Q$ is bipartite, we always know that the quiver with relations of $(\Delta A)^! \simeq \Pi_3 A$ can be given by the description in \cite[Theorem C]{Thi20} or \cite[Theorem A]{GI20} as $A$ is $2$-representation infinite and Koszul. This makes it easier to compute the (graded) center of $\Pi_3 A$ and check whether $A$ is $2$-representation tame and thus, equivalently, whether $\Delta A$ satisfies the (\textbf{Fg})-condition. Following this approach, it is for instance relatively straightforward to check that $A$ is \mbox{$2$-representation} tame in the case where $Q$ is of type $\widetilde{\mathbb{A}}_n$. This can e.g.\ be seen by similar arguments as those used in \cite{Anundsen-Sandoey}.
\end{example}

We now consider another class of examples that are not necessarily Koszul in the classical sense. A \emph{dimer algebra} is an infinite dimensional algebra derived from a dimer model on the torus; see e.g.\ \cite{Broomhead'12, Nakajima}. Namely, it is the Jacobian algebra $J(Q,W)$ of a quiver $Q$ with potential $W$ obtained from a bipartite graph that tiles the torus. A dimer model is said to be \textit{consistent} if there is a positive grading defined on the arrows of the associated quiver satisfying certain technical conditions; see e.g.\ \cite[Definition 2.2]{Nakajima}. Note that there are many equivalent notions of consistency appearing in the literature; see \cite{Boc12,IU11}. When a dimer model is consistent, it must have a \emph{perfect matching} by e.g.\ \cite[Proposition 8.1]{IU15}, meaning that there exists a subset of the edges of the given bipartite graph on the torus such that every vertex of the graph lies on exactly one edge of that subset. Moreover, such a perfect matching induces a positive grading on the dimer algebra; see the discussion after \cite[Definition 1.1]{Nakajima}.

It should be noted that $3$-preprojective algebras of so-called $2$-representation infinite algebras of type $\widetilde{\mathbb{A}}_n$, as introduced in \cite{HIO14}, are examples of dimer algebras coming from consistent dimer models by e.g.\ \cite[Section 5]{HIO14} or \cite{Dramburg-Gasanova24}. These higher preprojective algebras are Koszul in the classical sense, since they arise from skew group algebras of the polynomial ring in three variables. However, even though the class of dimer algebras associated to consistent dimer models consequently contains certain algebras that are classically Koszul, such algebras are not Koszul in general. Nevertheless, our theory allows us to easily deduce that if $\Gamma$ is a dimer algebra with a consistent dimer model and $A \coloneq \Gamma_0$ is finite dimensional, then $\Delta A$ satisfies the (\textbf{Fg})-condition.

\begin{thm}\label{thm: Fg and consistent dimer algebras}
Let $\Gamma$ be a dimer algebra associated to a consistent dimer model, and assume that the dimer model has a perfect matching inducing a grading such that $A \coloneq \Gamma_0$ is finite dimensional. Then $\Delta A$ satisfies the \emph{(\textbf{Fg})}-condition. 
\end{thm}

\begin{proof}
Given our assumptions, \cite[Proposition 3.5]{Nakajima} yields that $A$ is $2$-representation infinite. Moreover, as pointed out at the end of \cite[Section 3]{Nakajima}, we know that $A$ is in fact \mbox{$2$-representation} tame since $\Gamma$ is a non-commutative crepant resolution of its center. The result now follows by applying \cref{cor: fg}.
\end{proof}

We illustrate the theorem above with an example from \cite{Amiot-Iyama-Reiten}. 

\begin{example}
Consider the quiver $Q$ given by
\[
\begin{tikzcd}
1 \rar[shift left, "x_1"] \rar[shift right, "y_1" below] & 2 \dar[shift left, "x_2"] \dar[shift right, "y_2" left] \\
4  \uar[shift left, "x_4"] \uar[shift right, "y_4" right] & 3 \lar[shift left, "x_3"] \lar[shift right, "y_3" above]
\end{tikzcd}
\]
with potential $W = x_1 x_2 x_3 x_4 + y_1 y_2 y_3 y_4 - x_1 y_2 x_3 y_4 - y_1 x_2 y_3 x_4$. The relations obtained as $\partial_\alpha W$ for $\alpha \in Q_1$ are cubic, implying that  
\[
\Gamma \coloneq kQ/\langle \partial_\alpha W \, \lvert \, \alpha \in Q_1\rangle
\]
cannot be Koszul in the classical sense \cite[Proposition 1.2.3]{Beilinson-Ginzburg-Soergel}. By \cite[Examples 6.2]{Amiot-Iyama-Reiten}, we have that $\Gamma$ is the dimer algebra of a consistent dimer model with a perfect matching inducing a grading such that $A \coloneq \Gamma_0$ is finite dimensional. Moreover, one obtains such a grading by putting e.g.\ $\{x_4, y_4\}$ in degree $1$, in which case $A$ can be given by the quiver
\[
\begin{tikzcd}
1 \rar[shift left, "x_1"] \rar[shift right, "y_1" below] & 2 \rar[shift left, "x_2"] \rar[shift right, "y_2" below] & 3 \rar[shift left, "x_3"] \rar[shift right, "y_3" below] & 4 
\end{tikzcd}
\]
with relations $x_1 x_2 x_3 - y_1 x_2 y_3$ and $y_1 y_2 y_3 - x_1 y_2 x_3$. 
Using e.g.\ \cite{Schroer'99} or \cite{Fernandez-Platzeck'02,Fernandez-et-al'22}, one could compute the quiver and relations of $\Delta A$ explicitly, but for our purposes it suffices to observe that also $\Delta A$ cannot be Koszul in the classical sense since $A$ has cubic relations. We note that it seems quite difficult to compute both the $\Ext$-algebra of the simple modules and the Hochschild cohomology of $\Delta A$ and use this to verify the (\textbf{Fg})-condition directly. Nevertheless, we know from \cref{thm: Fg and consistent dimer algebras} that $\Delta A$ must satisfy the (\textbf{Fg})-condition. 
\end{example}

\begin{remark}\label{rem: n-RF case}
By \cite{Iyama-et-al}, it is known that $n$-representation finite algebras have trivial extensions that are twisted periodic, meaning that the simple modules have periodic projective resolutions, or equivalently that the algebra considered as a bimodule is isomorphic to one of its syzygies twisted by an automorphism on one side. A twisted periodic algebra is called periodic if the aforementioned automorphism can be chosen to be the identity. The periodicity conjecture of Erdmann and Skowronsk\'i claims that all twisted periodic algebras are in fact periodic \cite{Erdmann-Skowronski'15}. Using the ideas in \cite{Green-et-al-2}, one can check that twisted periodic algebras satisfy the (\textbf{Fg})-condition if and only if they are periodic. The periodicity conjecture thus suggests that trivial extensions of $n$-representation finite algebras are a source of algebras that satisfy the (\textbf{Fg})-condition. 
\end{remark}

\begin{acknowledgements}
The second author is grateful to have been supported by the Norwegian Research Council project 301375, ``Applications of reduction techniques and computations in representation theory'' during most of the work on this paper and by Norwegian Research Council project 357034, ``Symmetries of derived module categories via higher $\tau$-tilting- and almost $T$-Koszul theory'' while revising the paper in connection with the publication process. The first author acknowledges support from the project ``Pure Mathematics in Norway'' funded by the Trond Mohn Research Foundation, as well as from the Alexander von Humboldt Foundation.

The authors profited from use of the software QPA \cite{QPA} to compute examples which motivated parts of the paper. The authors would moreover like to thank \O yvind Solberg, Steffen Oppermann and Gustavo Jasso for helpful discussions.
\end{acknowledgements}

\bibliographystyle{plain}
\bibliography{bib.bib} 

\end{document}